\documentclass[a4paper, 12pt, reqno]{amsart}
\usepackage{amsmath,amsfonts,amssymb,graphics,mathrsfs,amsthm,eurosym,verbatim,enumerate,quotes,graphicx}
\usepackage[marginratio=1:1,tmargin=117pt, height=650pt]{geometry}
\usepackage{hyperref}
\usepackage{subfig}
\def\tef{transcendental entire function}
\def\qfor{\quad\text{for }}
\def\hmin{h_{\operatorname{min}}}
\def\endpoints{E}

\def\No{\mathbb{N}_0}
\def\addr{\operatorname{addr}}

\usepackage[flushmargin]{footmisc}
\usepackage[noadjust]{cite}
\makeatletter
\def\blfootnote{\xdef\@thefnmark{}\@footnotetext}
\makeatother
\newtheorem{thmx}{Theorem}

\theoremstyle{plain}
\newtheorem{theorem}{Theorem}[section]
\newtheorem{proposition}[theorem]{Proposition}

\newtheorem{lemma}[theorem]{Lemma}
\newtheorem{observation}[theorem]{Observation}
\newtheorem{definition}[theorem]{Definition}

\newtheorem*{remarks}{Remarks}

\newcommand{\C}{{\mathbb{C}}}

\newcommand{\B}{\mathcal B}
\newcommand{\R}{{\mathbb{R}}}
\newcommand{\Z}{{\mathbb{Z}}}
\newcommand{\N}{{\mathbb{N}}}

\newcommand{\s}{\underline{s}}

\newcommand{\Chat}{\hat{\mathbb{C}}}
\DeclareMathOperator*{\esssup}{ess\,sup}
\newcommand*{\defeq}{\mathrel{\vcenter{\baselineskip0.5ex \lineskiplimit0pt
                     \hbox{\scriptsize.}\hbox{\scriptsize.}}}%
                     =}
\begin{document}
\title[Dynamics of generalised exponential maps]{Dynamics of generalised exponential maps}
\author{Patrick Comd\"{u}hr, \, Vasiliki Evdoridou, \, David J. Sixsmith}

\address{Mathematisches Seminar \\ Christian-Albrechts-Universit\"{a}t zu Kiel \\ Ludewig-Meyn-Str. 4 \\D--24098 Kiel \\ Germany\textsc{\newline \indent \href{https://orcid.org/0000-0003-1811-7840}{\includegraphics[width=1em,height=1em]{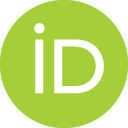} {\normalfont https://orcid.org/0000-0003-1811-7840}}}}
\email{comduehr@math.uni-kiel.de}

\address{School of Mathematics and Statistics\\ The Open University\\
Milton Keynes MK7 6AA\\ UK\textsc{\newline \indent \href{https://orcid.org/0000-0002-5409-2663}{\includegraphics[width=1em,height=1em]{orcid2.png} {\normalfont https://orcid.org/0000-0002-5409-2663}}}}
\email{vasiliki.evdoridou@open.ac.uk}

\address{Dept. of Mathematical Sciences \\ University of Liverpool \\ Liverpool L69 7ZL \\ UK \textsc{\newline \indent \href{https://orcid.org/0000-0002-3543-6969}{\includegraphics[width=1em,height=1em]{orcid2.png} {\normalfont https://orcid.org/0000-0002-3543-6969}}}}
\email{djs@liverpool.ac.uk}

\thanks{The second author was supported by Engineering and Physical Sciences Research Council grant EP/R010560/1. \vspace{3pt}\vspace{3pt}\vspace{3pt}\\ }

\begin{abstract}
Since 1984, many authors have studied the dynamics of maps of the form $\mathcal{E}_a(z) = e^z - a$, with $a > 1$. It is now well-known that the Julia set of such a map has an intricate topological structure known as a \emph{Cantor bouquet}, and much is known about the dynamical properties of these functions.

In recent papers some of these ideas have been generalised to a class of quasiregular maps in $\mathbb{R}^3$, which, in a precise sense, is analogous to the class of maps of the form $\mathcal{E}_a$. Our goal in this paper is to make similar generalisations in $\mathbb{R}^2$. In particular, we show that there is a large class of continuous maps, which, in general, are not even quasiregular, but are closely analogous to the map $\mathcal{E}_a$, and have very similar dynamical properties. In some sense this shows that many of the interesting dynamical properties of the map $\mathcal{E}_a$ arise from its elementary function theoretic structure, rather than as a result of analyticity.
\end{abstract}
\maketitle
\section{Introduction}
Let $f : \R^n \to \R^n$ be a function, and let $f^n$ denote the $n$th iterate of $f$. In this paper we are interested in the iteration of a continuous function $f:\R^2 \to \R^2$, which need not be analytic, and throughout we identify $\R^2$ with the complex plane $\C$ in the obvious way. A special case of such a function is when $f:\C \to \C$ is transcendental entire. Then we define the \textit{Julia set} $J(f)$ as the set of points $z \in \C$ where the iterates $\{f^n\}_{n \in \mathbb{N}}$ fail to form a normal family in any neighbourhood of $z$; roughly speaking, the iterates of $f$ are chaotic near a point in the Julia set. For an introduction to the properties of the Julia set, and the dynamics of {\tef}s, see, for example, \cite{bergweiler93} and \cite{Dierk}. 

In the study of the dynamics of {\tef}s, many authors have considered maps of the form 
\begin{equation}
\label{Eadef}
\mathcal{E}_a(z) \defeq e^z - a, \qfor a > 1.
\end{equation}
It is straightforward to show that $\mathcal{E}_a$ has an attracting fixed point $\xi \in \R$. We denote by $F$ the set of points that are attracted to $\xi$; in other words
\[
F \defeq \{ z \in \C : \mathcal{E}_a^n(z) \to \xi \text{ as } n \rightarrow\infty \}.
\]
It can be shown that $J(\mathcal{E}_a) = \C \setminus F$ . (Clearly $F$ here is the \emph{Fatou set} of $\mathcal{E}_a$, but we do not use this fact.)

The first study of the dynamics of maps of the form \eqref{Eadef} was by Devaney and Krych \cite{DevaneyandKrych}. Many authors since then have investigated these maps, and in the following we summarise some of the most important results that are known concerning their dynamical properties. Before stating the result we need a number of definitions.

We say that a component $\gamma$ of $J(\mathcal{E}_a)$ is a \emph{Devaney hair} if it is a simple curve $\gamma : [0, \infty) \to \C$ with the properties that:
\begin{enumerate}[(I)]
\item $\gamma(t) \rightarrow \infty$ as $t \rightarrow \infty$.\label{it:dev1}
\item For each $n \geq 0$, $\mathcal{E}_a^n(\gamma)$ is a simple curve that connects $\mathcal{E}_a^n(\gamma(0))$ to $\infty$. We call $\gamma(0)$ the \emph{endpoint} of the curve $\gamma$.\label{it:dev3}
\item For each $t>0$, $\mathcal{E}_a^n \rightarrow\infty$ as $n\rightarrow\infty$ uniformly on $\gamma([t,\infty))$.\label{it:dev2}
\end{enumerate}
Note that there are other definitions of a Devaney hair in the literature; we have used the definition first used in \cite{DevHairs}, although, unlike in that paper, we do not formally specify that the hairs lie in the Julia set. 

A subset of $\C$ is a \emph{Cantor bouquet} if it is ambiently homeomorphic to a topological object known as a \emph{straight brush}; see \cite{AartsandO} for a precise definition. We say that $X \subset \Chat$ is \emph{totally separated} if for all $a, b \in X$, with $a\neq b$, there exists a relatively open and closed set $U \subset X$ such that $a \in U$ and $b \notin U$.

We are now able to state the results.
\begin{thmx}
\label{theo:exp}
Let $\mathcal{E}_a$ be the {\tef} defined in \eqref{Eadef}. Then the following all hold.
\begin{enumerate}[(a)]
\item $J(\mathcal{E}_a)$ has uncountably many components, each of which is a Devaney hair.\label{hairs}
\item $J(\mathcal{E}_a)$ is a Cantor bouquet.\label{bouquet}
\item Write $\endpoints$ for the set of endpoints of the Devaney hairs in $J(\mathcal{E}_a)$. Then $\endpoints$ is totally separated, but $\endpoints \cup \{\infty\}$ is connected.\label{separated}
\end{enumerate}
\end{thmx}
\begin{remarks}\normalfont
Part \eqref{hairs} seems to be a combination of results from \cite{DevaneyandTangerman}, \cite{Karpinska2} and \cite{Lasse}. Part \eqref{bouquet} is a result of \cite{AartsandO}; although the term ``Cantor bouquet'' had been used previously, this was the first paper to give a precise topological definition of such an object. 
Part \eqref{separated} is from \cite{Mayer}; this result can also be stated that $\infty$ is an \emph{explosion point} for the set $\endpoints \cup \{\infty\}$. 
Note that many of the authors cited above considered, in fact, the {\tef}s $\tilde{\mathcal{E}}_\lambda(z) \defeq \lambda e^z$, for $\lambda \ne 0$. The functions $\mathcal{E}_a$ and $\tilde{\mathcal{E}}_{e^{-a}}$ have the same dynamics, as they are conjugate via the map $z \mapsto z - a$. 
\end{remarks}

Our goal in this paper is to show that there is a large class of continuous functions $f : \R^2 \to \R^2$ that are analogous to the map $\mathcal{E}_a$ and also have the dynamical properties listed above. This shows that, in some sense, the properties listed in Theorem~A derive from elementary function theoretic properties of $\mathcal{E}_a$ rather than its analyticity. We stress that the functions in our class are continuous but not necessarily quasiregular.

To define our maps, we observe that if $z = x + iy$, then $e^z = e^x(\cos y + i \sin y)$. So we begin by considering a map
\begin{equation*}
Z : \C \to \C; \quad Z(x+ iy) \defeq g(x) h(y),
\end{equation*} 
where $g : \R \to \R$ and $h : \R \to \C$ are continuous functions defined in such a way that $g$ has behaviour analogous to the real exponential function, and $h$ has behaviour analogous to the map $y \mapsto \cos y + i \sin y$.

We first specify the map $h$. Let $\eta$ be a simple curve from $-i$ to $i$ with the following two properties; see Figure~\ref{f1}. Firstly, we have that 
\[
\eta\setminus\{-i, i\} \subset \{ z \in \C : \operatorname{Re} z >0 \text{ and } |z| \leq 1 \}.
\]
Secondly, for each $\theta \in [-\pi/2, \pi/2]$, the curve $\eta$ contains exactly one point of argument $\theta$. We then let $h : [-\pi/2, \pi/2] \to \eta$ be a biLipschitz map, such that $h(\pm \pi/2) =  \pm i$. Finally we extend $h$ to a map from $\R$ to $\eta \cup -\eta$ as follows; if $y \in \R$, and $y' \in [-\pi/2, \pi/2]$ is such that $y' = y + p\pi$, for some $p \in \Z$, then 
\[
h(y) \defeq 
\begin{cases}
h(y'),  &\text{ for } p \text{ even}, \\
-h(y'), &\text{ for } p \text{ odd}.
\end{cases}
\]

\begin{figure}
	\includegraphics[width=12cm,height=8cm]{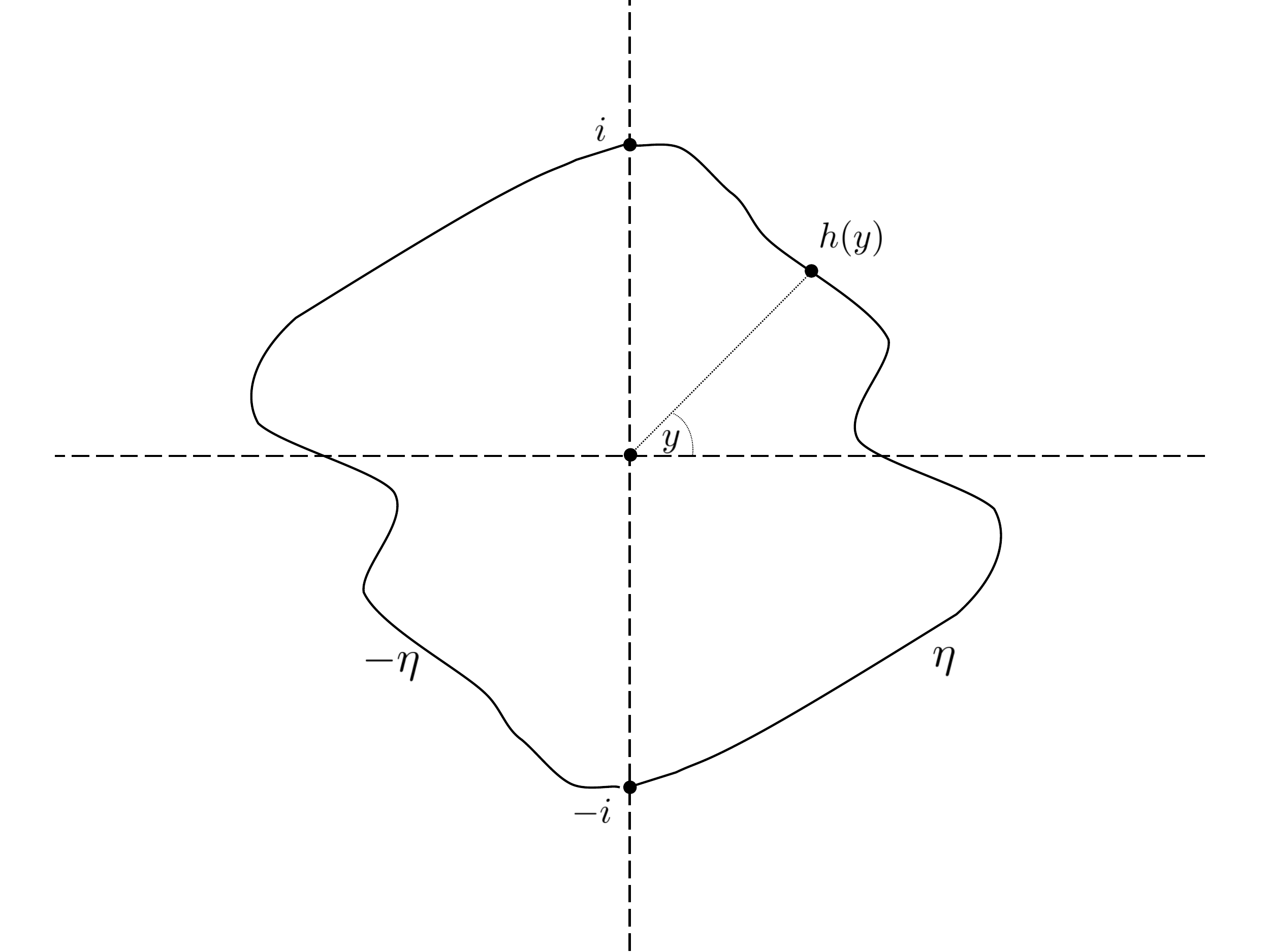}
  \caption{An illustration of a definition of the function $h$.}\label{f1}
\end{figure}

It is also useful to define $h(y) \defeq h_1(y) + i h_2(y)$, for real valued functions $h_1$ and $h_2$. Note also that it follows from the above that there exists $\hmin \in (0,1)$ such that
\begin{equation}
\label{hdef}
\hmin \leq |h(y)| \leq 1, \qfor y \in \R.
\end{equation}

Next we specify the function $g$. We let $g : \R \to (0, \infty)$ be such that the following conditions both hold.
\begin{enumerate}[(A)]
\item The function $g$ is strictly increasing, and $g(x) \rightarrow 0$ as $x \rightarrow -\infty$.\label{fun.A}
\item The derivative $g'(x)$ is defined almost everywhere, and is monotonically non-decreasing where it is defined.\label{fun.B}
\end{enumerate}
Note that it follows that, where defined, $g'(x)$ tends to $0$ as $x$ tends to $-\infty$. Note also that the positive, real, convex functions $G$ with $G(x) \to 0$ as $x\to -\infty$ is a large class of functions which satisfy properties \eqref{fun.A} and \eqref{fun.B}.

We need $g$ to grow sufficiently quickly. In particular, we suppose that there exists $c > 1/\hmin > 1$ such that,
\begin{equation}
\label{gconstraint}
g(x + 2\pi) \geq c g(x), \qfor x \text{ sufficiently large}.
\end{equation}

We then define the map $Z$ mentioned earlier by
\begin{equation}
\label{eq:Zdef}
Z(x+iy) \defeq g(x) h(y) = g(x) h_1(y) + i g(x) h_2(y).
\end{equation}
Note that when $g(x) = e^x$, the function $Z$ in \eqref{eq:Zdef} is quasiregular, and is known as a \emph{Zorich map}. If, in addition, $h(y) = \cos y + i \sin y$, then we have $Z \equiv \exp$.

Finally we define the function we are going to iterate. We let $a > 0$ and set
\begin{equation}
\label{eq:fdef}
f(z) \defeq Z(z) - a.
\end{equation}
We will later ensure that $a$ is sufficiently large for various conditions to hold. We then make the following definition. 
\begin{definition}
Suppose that $f$ is as defined in \eqref{eq:fdef}, where $Z$ is as defined in \eqref{eq:Zdef} for functions $g, h$ that satisfy all the conditions listed earlier. Then we say that $f$ is a \emph{generalised exponential}.
\end{definition}
Although not necessarily quasiregular, $f$ is continuous, open, discrete, and differentiable almost everywhere. Note also that it follows from these definitions that any local inverse of $f$ is also continuous and differentiable almost everywhere. The name for the generalised exponentials can be further justified by equation \eqref{eq:ggrowth},  given in Section~\ref{sec.basic}, which shows, in particular, that no polynomial can satisfy condition \eqref{gconstraint}.

For a generalised exponential, $f$, there is no obvious definition of a Julia set; although the Julia set can be defined for quasiregular maps \cite{WalterDan}, we do not want to assume that $f$ is even quasiregular. The following result allows us, nonetheless, to establish a set analogous to the Julia set. Here we define 
\[
\mathbb{H}_r \defeq \{ x + iy : x \leq r \}, \qfor r \in \R.
\]
\begin{theorem}
\label{thereisaFatouset}
Suppose that $f$ is a generalised exponential. Then, there exist $m<0$ and $M>0$ such that whenever $a$ is sufficiently large, $f$ has a unique attracting fixed point $\xi \in \mathbb{H}_m$, $f(\mathbb{H}_M) \subset \mathbb{H}_m$, and all points of $\mathbb{H}_M$ tend to $\xi$ under iteration.
\end{theorem}
We can now use Theorem~\ref{thereisaFatouset} to make the following natural definition.
\begin{definition}
\label{def:JandF}
If the conditions of Theorem~\ref{thereisaFatouset} hold, then we let $F$ denote the set of points that iterate to the unique attracting fixed point, and set $J \defeq \C \setminus F$.
\end{definition}

Our main result is then an extension of Theorem~A to the class of generalised exponentials.
\begin{theorem}
\label{theo:main}
Suppose that $f$ is a generalised exponential. Then, for all sufficiently large values of $a$, the following all hold.
\begin{enumerate}[(a)]
\item $J$ has uncountably many components, each of which is a Devaney hair.\label{Jhairs}
\item $J$ is a Cantor bouquet.\label{Jbouquet}
\item If $\endpoints$ is the set of endpoints of the Devaney hairs in $J$, then $\endpoints$ is totally separated, but $\endpoints \cup \{\infty\}$ is connected.\label{Jseparated}
\end{enumerate}
\end{theorem}
Note that the fact that $J$ has uncountably many components is also a consequence of \eqref{Jbouquet}. However, it seems worth emphasising this fact.
\subsection*{Structure}
The structure of the paper is as follows. First in Section~\ref{sec.basic} we prove Theorem~\ref{thereisaFatouset}. The proof of Theorem~\ref{theo:main} is then spread across the rest of the paper. 

%
%
%
%
%
\section{Existence of the sets F and J}
\label{sec.basic}
In this section we give the proof of Theorem~\ref{thereisaFatouset}, and so establish the existence of the sets $F$ and $J$ from Definition~\ref{def:JandF}. Firstly, we need a form of expansion for $Z$, which is given in the following lemma.
\begin{lemma}
\label{lemm:deriv}
There exist constants $\mu > 1$ and $M > 0$ such that
\begin{equation}
\label{deriveq}
\ell(DZ(z)) \geq \mu > 1, \quad \text{ a.e. for } \operatorname{Re} z \geq M.
\end{equation}	
\end{lemma}
\begin{proof}
Let $c>1$ be the constant from \eqref{gconstraint}. We claim first that there exist $C>0$ and $x_0>0$ such that
\begin{equation}
\label{eq:ggrowth}
g(x) \geq C c^{\frac{x}{2\pi}}, \qfor x \geq x_0.
\end{equation}
To prove this claim, choose $x_0$ sufficiently large that \eqref{gconstraint} holds for $x \geq x_0$. Set
\[
C \defeq \max \{ r \in \R : g(x) \geq r c^{\frac{x}{2\pi}}, \text{ for }x \in [x_0, x_0 + 2\pi] \}.
\]
It is easy to see that this maximum exists and is positive. Now, suppose that $x \geq x_0$. Set $x = x' + 2\pi k'$, where $x' \in [x_0, x_0 + 2\pi)$ and $k'$ is a non-negative integer. Then, by repeated application of \eqref{gconstraint}, and by the definition of $C$,
\[
g(x) \geq c^{k'} g(x') \geq C c^{k'} c^{\frac{x'}{2\pi}} = C c^{\frac{x}{2\pi}}.
\]
This equation completes the proof of our first claim.

We also require a growth condition on $g'$. We claim that there exists $x_0'>0$ such that
\begin{equation}
\label{eq:g'growth}
g'(x)\geq \frac{c-1}{2\pi} g(x-2\pi), \quad \text{ a.e. for } x \geq x_0'.
\end{equation}

To prove this claim, note that since $g'$ is monotonically non-decreasing where defined, we have that for sufficiently large values of $x$,
	\begin{equation*}
	g(x+2\pi)-g(x)=\int_{x}^{x+2\pi} g'(t)\, dt \leq 2\pi g'(x+2\pi), \quad \text{ a.e. }
	\end{equation*}
	Combining this with
	\begin{equation*}
	g(x+2\pi)-g(x) \geq (c-1)g(x),
	\end{equation*}
	we obtain the result.

For the derivative of $Z$ we have, whenever it exists, that
\begin{equation*}
DZ(x+iy)= \left(\begin{array}{cc}%
g'(x) h_1(y) & g(x) h_1'(y)  \\
g'(x) h_2(y) & g(x) h_2'(y)
\end{array}\right)
= \left(\begin{array}{cc}%
h_1(y) & h_1'(y)  \\
h_2(y) & h_2'(y)
\end{array}\right) \cdot 
\left(\begin{array}{cc}%
g'(x)  & 0  \\
  0 & g(x)
\end{array}\right).
\end{equation*}
The fact that $h$ is biLipschitz yields that $h'$ exists almost everywhere.
Since $|h(y)| \geq \hmin > 0$, where the derivative exists we obtain
\begin{align*}
\inf\limits_{|w| = 1} |DZ(x,y)\cdot w| 
&=  \inf\limits_{|w| = 1} \left| \left(\begin{array}{cc}%
h_1(y) & h_1'(y)  \\
h_2(y) & h_2'(y)
\end{array}\right) \cdot 
\left(\begin{array}{cc}%
g'(x)  & 0  \\
0 & g(x)
\end{array}\right) \cdot w \right| \\
&\geq \min\{g(x),g'(x)\} \inf\limits_{|w| = 1} \left|\left(\begin{array}{cc}%
h_1(y) & h_1'(y)  \\
h_2(y) & h_2'(y)
\end{array}\right) \cdot w \right| \\
&\geq c_h\cdot\min\{g(x),g'(x)\},
\end{align*}
for a suitable constant $c_h >0$, which depends only on the Lipschitz constant of $h$; see \cite[Section 2]{Bergweiler}. Since $c>1$ the result then follows from \eqref{eq:ggrowth} and \eqref{eq:g'growth}.
\end{proof}
\begin{proof}[Proof of Theorem~\ref{thereisaFatouset}]
Let $M>0$ be the constant from Lemma~\ref{lemm:deriv}. Since $h$ is biLipschitz, there exists a constant $L\geq 1$ such that 
\[
|h'(y)|\leq L, \quad\text{a.e. for } y\in \R.
\]
Since $g(x)$ and $g'(x)$ (where defined) both tend to $0$ as $x$ tend to $-\infty$, we can choose $m<0$ sufficiently small that 
\[
g(x)+g'(x) \leq \frac{1}{2(1+L)}, \quad\text{a.e. for } x \leq m.
\]
We deduce that
\begin{align*}
||DZ(x+iy)|| &\leq |g'(x) h_1(y)| + |g(x) h_1'(y)| + |g'(x) h_2(y)| + |g(x) h_2'(y)| \\
             &\leq (L+1)(g'(x)+g(x)) \\
		   		 	 &\leq \frac{1}{2}, \quad\text{a.e. for } x \leq m.
\end{align*}
Since $DZ(x+iy)=Df(x,y)$ it follows that 
\[
|f(z_1) - f(z_2)| \leq \frac{1}{2}|z_1 - z_2|, \qfor z_1, z_2 \in \mathbb{H}_m.
\]

Now choose 
\begin{equation}
\label{adef}
%
%
a > \max \{0, g(M) - m, g(M) - M \}.
\end{equation}
(Note that the choice of $a$ here is stronger than is required in this proof, but convenient for use in later results).

If $x \leq M$, then
\[
\operatorname{Re} f(x+iy) = g(x) h_1(y) - a \leq g(M) - a \leq m.
\]
In other words, $f(\mathbb{H}_M) \subset \mathbb{H}_m$. Hence $f$ is a contraction mapping on $\mathbb{H}_m$, and so $\mathbb{H}_m$ contains a unique attracting fixed point $\xi$ by the Banach fixed point theorem. Since $f$ is expanding in the complement of $\mathbb{H}_M$, by Lemma~\ref{lemm:deriv}, the uniqueness of $\xi$ is immediate.
\end{proof}
In the remainder of the paper we will assume that $f, g, h, Z$ are as defined above, that $f$ is a generalised exponential, and that $a$ has been chosen such that \eqref{adef} holds.

 \begin{figure}
  \subfloat{\includegraphics[width=.45\textwidth]{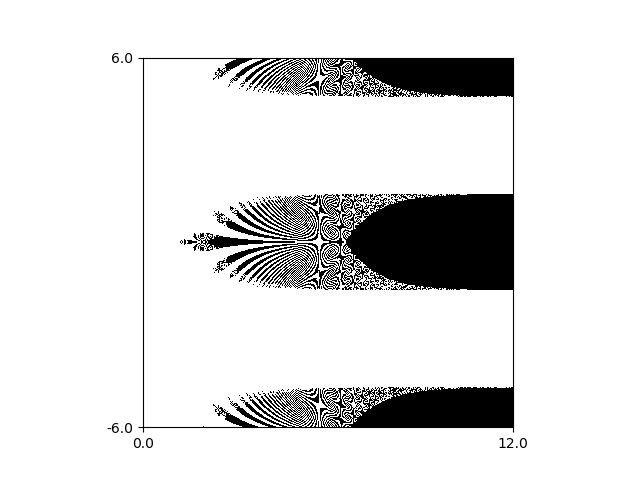}}\hfill
  \subfloat{\includegraphics[width=.45\textwidth]{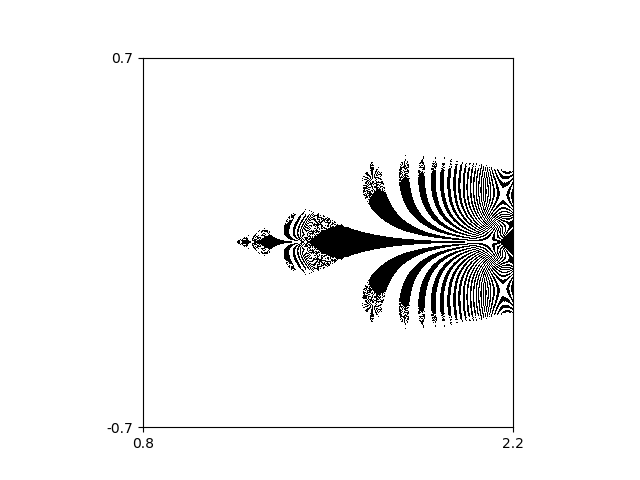}}
   \caption{\label{f2}An illustration of $J$ for the generalised exponential with $a=2$, with $g(x) = e^x$, and with $h$ being the obvious extension of the two linear maps from $[-1, 0]$ to $[-i, 1]$ and from $[0, 1]$ to $[1, i]$.}
 \end{figure}


\begin{figure}
  \subfloat{\includegraphics[width=.45\textwidth]{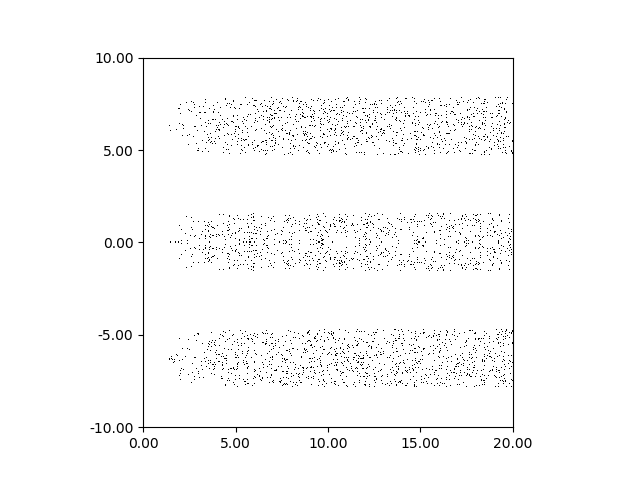}}\hfill
  \subfloat{\includegraphics[width=.45\textwidth]{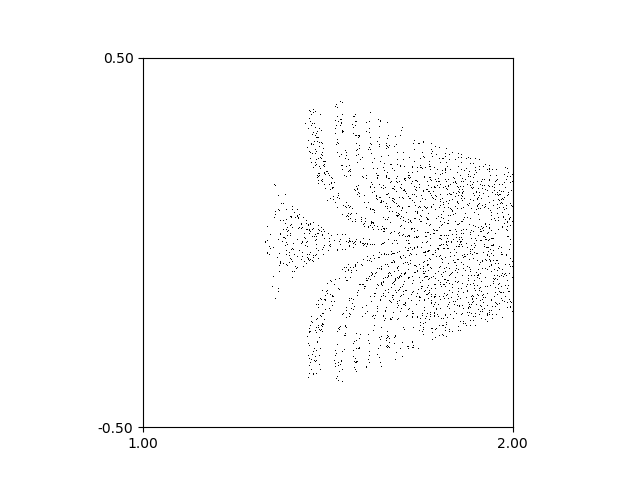}}
   \caption{\label{f3}An illustration of $J$ for the function with $a=1$, with $h(y) = (\cos y, \sin y)$, and with $g(x) = 0$ in the left half-plane and $g(x) = x^3$ in the right half-plane. Note that this function is \emph{not} covered by the results of this paper. However, it is still possible to define $J$ in a similar way, and we still observe some of the features of $J$ that we might expect.}
 \end{figure}

\section{Symbolic dynamics}
\label{sec.symbolicdynamics}
In this section we define tracts and external addresses, and then use these to establish symbolic dynamics on $J$. We begin by defining the tracts of the function $g$. Since $-a < 0 < M$, we have that $\mathbb{H}_{-a} \subset \mathbb{H}_M$, and so points with imaginary part in an interval $[(4k+1)\pi/2, (4k+3)\pi/2]$, for some $k \in \Z$, necessarily lie in $F$. 

\begin{definition}
\label{def:various}
For each $k \in \Z$, define the \emph{tract} $T_k$ by
\[
T_k \defeq \left\{ x + iy \in \C : x > M \text{ and } \frac{(4k-1)\pi}{2} < y < \frac{(4k+1)\pi}{2} \right\}.
\]
Also set
\begin{equation}
\label{eq:Hdef}
H \defeq f(T_0).
\end{equation}
\end{definition}
Clearly $H$ is the image of any tract. Geometrically $H$ is the right half-plane with a bounded set removed; in particular
\[
H = \{ z \in \C : \operatorname{Re} z > a \} \setminus \overline{f(\{ x + iy \in \C : x \in (0, M) \text{ and } y \in (-\pi/2, \pi/2)\})}.
\]

We stress that the sets $T_k$ are not tracts in the sense usually defined for functions in the class $\B$. However, if $T_k, T_{k'}$ are both tracts, then it follows by \eqref{adef} that $\overline{T_{k'}} \subset f(T_k) = H$; abusing slightly the terminology of class $\B$ maps, $f$ is of \emph{disjoint type}. 

Note also that if $T_k$ is a tract, then $f : T_k \to H$ is a continuous bijection, and in fact the same is true for $f : \overline{T_k} \to \overline{H}$. (This follows from the definitions of $g$ and $h$; in fact $f$ is a bijection on a set slightly larger than $T_k$.) It follows that $f : T_k \to H$ is a homeomorphism. We denote the inverse of this restriction by $f_k^{-1}$.

More generally, if $k_1 k_2 \ldots k_n$ is a finite sequence of integers, then we define $f_{k_1 k_2 \ldots k_n}^{-n} \defeq f_{k_1}^{-1} \circ \ldots \circ f_{k_n}^{-1}$.

Next we consider the components of $J$ and define the notion of external addresses. 
\begin{proposition}
\label{prop:unbounded}
Every component of $J$ is unbounded.
\end{proposition}
\begin{proof}
Note that if $Y \subset H$ is connected and unbounded, and $k \in \Z$, then $f_{k}^{-1}(Y)$ is connected and unbounded; connectedness follows from continuity, and unboundedness is a consequence of the fact that $f$ is a homeomorphism of the closure of each tract.

For each $n \in \N$, consider the set 
\[
F_n = \bigcup_{k_1 \ldots k_n \in \Z^n} f^{-n}_{k_1 \ldots k_n}(H) \cup \{\infty\}.
\]

Then, considered as a subset of the Riemann sphere, $F_n$ is connected, by the above remark, and compact; in other words, $F_n$ is a continuum. It follows that $$J \cup \{\infty\} = \bigcap_{n \in \N} F_n$$ is a nested intersection of continua, and so is itself a continuum. The result then follows by the ``Boundary bumping theorem''; see, for example, \cite[Theorem 5.6]{Nadler}.
\end{proof}

We write $\No = \N \cup \{0\}$.
\begin{definition}
\label{def:admissible}
Observe that $z \in J$ if and only if there is an \emph{external address} $\s = s_0 s_1 \ldots \in \Z^{\No}$ such that $f^n(z) \in T_{s_n}$, for $n \geq 0$. We write $s = \addr(z)$. We also write $J_{\s}$ for the set of points with external address $\s$. Finally we let $\hat{J_{\underline{s}}}$ denote the closure of $J_{\s}$ in $\hat{\C}$. If $\s$ is an external address such that $J_{\s} \ne \emptyset$, then we say that $\s$ is \emph{admissible}.
\end{definition}

The conclusions of the following observation are straightforward, and the proof is omitted. Here $\sigma$ is the \emph{Bernoulli shift map} defined by $\sigma(s_0 s_1 s_2 \ldots) = s_1 s_2 \ldots$.
\begin{observation}
\label{obs:basics}
Suppose that $\s = s_0 s_1 \ldots$ and $\s'$ are admissible external addresses, with $\s \ne \s'$. Then all the following hold.
\begin{enumerate}
\item $J_{\s} = \bigcap_{n \in \N} f_{s_0 s_1 \ldots s_{n-1}}^{-n}(H)$. 
\item $J_{\s}$ and $J_{\s'}$ are disjoint.
\item If $n \in \N$, then $f^n(J_{\s}) = J_{\sigma^n(\s)}$.
\end{enumerate}
\end{observation}

Our next step is to characterise the admissible external addresses. 
\begin{definition}
We say that an external address $\s \in \Z^{\No}$ is \emph{$g$-bounded} if there exists $x_0 \geq 0$ such that
\begin{equation}
\label{gbound}
2\pi |s_n| \leq g^n(x_0), \qfor n \in \No.
\end{equation}
\end{definition}
Note that the constant $2\pi$ in \eqref{gbound} can, in fact, be replaced by any positive constant; indeed, this comment also applies to the choice of the constant $2\pi$ in the definition of admissible external addresses in \cite{DevaneyandKrych}. We have used $2\pi$ here for consistency.

We show that the admissible external addresses are identically the external addresses that are $g$-bounded, provided that $g$ satisfies (\ref{gconstraint}).
\begin{theorem}
\label{thm:gbounded-admissible}
Suppose that \eqref{gconstraint} holds, and that $\s \in \Z^{\No}$. Then $\s$ is admissible if and only if $\s$ is $g$-bounded.
\end{theorem}
\begin{proof}
First, suppose that $\s = s_0 s_1 \ldots$ is admissible, and so there exists a point $z = x + iy \in \C$ with external address $\s$. Note that 
\[
\operatorname{Im} f(z) \leq |Z(z)| = g(x)|h(y)| \leq g(x),
\]
and indeed 
\[
\operatorname{Im} f^n(z) \leq |Z^n(z)| \leq g^n(x), \qfor n \in \No.
\]
Observe that it follows from \eqref{gconstraint} that there exists $x_0' > 0$ such that
\[
g^n(x + 2\pi) \geq g^n(x) + 2 \pi, \qfor x \geq x_0', \ n \in \No.
\]
Hence, for $n \in \No$, 
\[
2\pi|s_n| \leq \operatorname{Im} f^n(z) + 2\pi \leq g^n(x) + 2\pi \leq g^n(x+x_0') + 2\pi \leq g^n(x + x_0'+ 2\pi),
\]
and so $\s$ is indeed $g$-bounded.

In the other direction, suppose that $\s = s_0 s_1 \ldots \in \Z^{\No}$ is $g$-bounded. Fix $\kappa > 0$ small enough that $$\kappa^2 < \frac{c^2}{\hmin^2} - 1.$$ Let $\tilde{g}(x) \defeq \hmin g(x)$. It can be deduced from \eqref{gconstraint} and \eqref{gbound} that there exists $x_0'' \geq 0$ such that
\begin{equation}
\label{gbound2}
\max\{3\pi/2, 4\pi |s_n|\} \leq \kappa \tilde{g}^n(x), \qfor x \geq x_0'' \text{ and } n \in \No.
\end{equation}
Choose $\delta > 0$ sufficiently small that $$(1+\delta)^2 < \frac{c^2}{\hmin^2} - \kappa^2,$$ and choose $$r_0 > \max\left\{M, x_0'', \frac{2\pi+a}{\delta}\right\}.$$ Increasing $r_0$, if necessary, we can also assume that all points of real part at least $r_0$ lie in $H$. We then set $r_{k+1} = \tilde{g}(r_k)$, for $k \in \No$.

For each $n \in \No$, let $D_n$ be the closed square of side $2\pi$, with sides parallel to the axes, and with bottom left vertex at the point $(r_n, (4s_n-1)\pi/2)$.

We claim that $f(D_n) \supset D_{n+1}$, for $n \in \No$. To prove the claim, first fix $n \in \No$. Note that, by a calculation, $f(D_n)$ contains the annulus
\[
A_n \defeq \{ z \in \C : \hmin g(r_n) \leq |z+a| \leq c g(r_n)\}.
\]

Since $\hmin g(r_n) = r_{n+1}$, we can see that $D_{n+1}$ does not lie inside the inner radius of $A_n$. Hence it remains to prove that $D_{n+1}$ does not lie outside the outer radius of this annulus. Without loss of generality we can assume that $s_{n+1}$ is non-negative. A furthermost point of $D_{n+1}$ from $(-a, 0)$ is the point $(r_{n+1} + 2\pi, (4s_{n+1}+3)\pi/2)$. Hence the square of the distance from $(-a, 0)$ to any point of $D_{n+1}$ is at most
\[
(r_{n+1} + a + 2\pi)^2 + \left(\frac{(4s_{n+1}+3)\pi}{2}\right)^2 \leq r_{n+1}^2(1+\delta)^2 + r_{n+1}^2 \kappa^2 \leq r_{n+1}^2 \cdot \frac{c^2}{\hmin^2},
\]
where we have used \eqref{gbound2}, together with the choices of $\delta$ and $\kappa$. Since the outer radius of $A_n$ is $c g(r_n) = c r_{n+1} / \hmin$, this completes the proof of the claim.

It follows by, for example, \cite[Lemma 1]{SlowEscaping}, that there is point $z$ such that $f^n(z) \in D_n$, for $n \in \No$. Since $D_n \setminus T_{s_n}$ maps to the complement of $H$, we in fact have that $f^n(z) \in T_{s_n}$, for $n \in \No$. In other words, $z \in J_{\s}$, which completes the proof.
\end{proof}

\section{Devaney Hairs}
\label{sec.hairs}
Our goal in this section is to show that each component of $J$ is a Devaney hair. Part \eqref{Jhairs} of Theorem~\ref{theo:main} follows, since there are uncountably many $g$-bounded, and hence admissible, external addresses. Note that this requires that we establish the three properties \eqref{it:dev1}, \eqref{it:dev3} and \eqref{it:dev2}.

We first show that our function $f$ satisfies a \emph{uniform head-start condition}; this terminology is from \cite{RRRS}. This is a key ingredient in the arguments we use in the remainder of this paper. 
%
%
%
We require the following expansion estimate, which follows from \eqref{deriveq}; recall that $\mu > 1$.
\begin{proposition}
\label{prop:expansion}
Suppose that $f$ is a generalised exponential function, that $n \in \N$, and that $U$ is a component of $f^{-n}(\C \setminus \mathbb{H}_M)$. Then
\[
|f^n(z) - f^n(w)| \geq \mu^n |z - w|, \qfor z, w \in U.
\] 
\end{proposition}
\begin{proof}
Let $\phi : \C\setminus\mathbb{H}_M \to U$ be the inverse to $f^n$. Since $\C\setminus\mathbb{H}_M$ is convex, it follows by \eqref{deriveq} that, if $z', w' \in \C\setminus\mathbb{H}_M$, then
\[
|\phi(z')-\phi(w')| \leq \int_{[z',w']} |\phi'(\zeta)| |d\zeta|
                    \leq |z'-w'| \esssup_{\zeta\in[z',w']} |D\phi(\zeta)|
										\leq \frac{1}{\mu^n} |z'-w'|,
\]
where $[z',w']$ denotes the line segment from $z'$ to $w'$. The result follows.
\end{proof}
We now prove that $f$ satisfies a uniform head-start condition.
\begin{lemma}
\label{lem:uniformheadstart}
Suppose that $f$ is generalised exponential function. Then there exists $K > 1$ with the following properties. 
\begin{enumerate}[(i)]
\item\label{it1} Suppose that $T, T'$ are tracts. If $z_0, z_1 \in \overline{T}$ and $f(z_0), f(z_1) \in \overline{T'}$, then
\[
\operatorname{Re} z_1 \geq K \operatorname{Re} z_0 \implies \operatorname{Re} f(z_1) \geq K \operatorname{Re} f(z_0).
\]
\item\label{it2} Suppose that $z_0, z_1$ have the same external address. Then there exist $k \in \N$ and $j \in \{0, 1\}$ such that
\[
\operatorname{Re} f^p(z_j) \geq K \operatorname{Re} f^p(z_{1-j}), \qfor p \geq k.
\]
\end{enumerate}
\end{lemma}
\begin{proof}
Note first that
\begin{equation}
\label{easy1}
|f(z)| - a \leq |Z(z)| \leq g(\operatorname{Re} z), \qfor z \in \C,
\end{equation}
and
\begin{equation}
\label{easy2}
|f(z)| + a \geq |Z(z)| \geq \hmin g(\operatorname{Re} z), \qfor z \in \C.
\end{equation}

First we prove \eqref{it1}. Suppose that $T, T'$ are two tracts, that $z_0, z_1 \in \overline{T}$, and that $f(z_0), f(z_1) \in \overline{T'}$. Suppose that $q \in \N$, that $K \geq 2 \pi q$, and finally that $\operatorname{Re} z_1 \geq K \operatorname{Re} z_2$. Then, by \eqref{gconstraint} and \eqref{easy2},
\begin{align*}
\operatorname{Re} f(z_1) &\geq |f(z_1)| - |\operatorname{Im} f(z_1)|, \\
                        &\geq \hmin g(\operatorname{Re} z_1) - a - |\operatorname{Im} f(z_0)| - \pi, \\
												&\geq \hmin c^q g(\operatorname{Re} z_0) - a - |f(z_0)| - \pi.
\end{align*}								

We then consider two possibilities. Suppose first that $|f(z_0)| \geq 2a$, so that, by \eqref{easy1}, $g(\operatorname{Re} z_0) \geq |f(z_0)| - a \geq \frac{1}{2}|f(z_0)|$. Then
\[
\operatorname{Re} f(z_1) \geq \left(\frac{1}{2}\hmin c^q - 1\right)|f(z_0)| - a - \pi.				
\]

On the other hand, if $|f(z_0)| < 2a$, then 
\[
\operatorname{Re} f(z_1) \geq \hmin c^q g(M) - 3a - \pi \geq \hmin c^q g(M) \frac{|f(z_0)|}{2a} - 3a - \pi.
\]

Since Re $f(z_0) \geq M$, the conclusion \eqref{it1} follows provided that $q$, and hence $K$, is chosen sufficiently large. (Note that the choice of $q$ can be made independently of $z_0$ and $z_1$.)

For \eqref{it2}, suppose that $z_0\ne z_1$ have the same external address. Fix $p \in \N$. Since $z_0$ and $z_1$ have the same external address, there exists a component $U$ of $f^{-p}(\C\setminus\mathbb{H}_M)$, containing both $z_0$ and $z_1$, that maps injectively to $\C\setminus\mathbb{H}_M$. It follows by Proposition~\ref{prop:expansion} that $|f^p(z_0) - f^p(z_1)| \geq \mu^p |z_0 - z_1|$.  The result then follows by \eqref{it1}, since $f^p(z_0)$ and $f^p(z_1)$ lie in the same tract, and $p$ was arbitrary.
\end{proof}
%
%
Next we use the uniform head-start condition to prove the existence of unbounded simple curves in $J$; in other words, we prove that $J$ consists of simple curves that satisfy \eqref{it:dev1} and \eqref{it:dev3}. We defer the proof of \eqref{it:dev2} until a little later.

Next we introduce a so-called \emph{speed ordering}. For each $z,w \in J_{\underline{s}}$ we say that $z \succ w$ if there exists $k \in \N$ with the property that $ \operatorname{Re}f^k(z) > K \operatorname{Re}f^k(w)$, where $K>1$ is the constant from Lemma~\ref{lem:uniformheadstart}. We extend this order to the closure of $J_{\s}$ in $\hat{\C}$, which we denote by $\hat{J_{\underline{s}}}$, by the convention that $\infty \succ z$ for all $z \in J_{\underline{s}}$. We then have the following.
\begin{lemma}
\label{lem: arcs in Js}
Suppose that $f$ is a generalised exponential function, and that $\s$ is an admissible external address. 
Then $(\hat{J_{\underline{s}}}, \succ)$ is a totally ordered space, and $J_{\underline{s}}$ has a unique unbounded component, which is a simple closed arc to infinity.
\end{lemma}
\begin{proof}
The fact that $(\hat{J_{\underline{s}}}, \succ)$ is a totally ordered space is a straightforward consequence of Lemma~\ref{lem:uniformheadstart}. 

We then claim that each component of $\hat{J_{\underline{s}}}$ is homeomorphic to a compact interval, which may be degenerate. The proof of this fact is exactly as in the proof of \cite[Proposition 4.4(a)]{RRRS}; it is first shown that the identity map from $\hat{J_{\s}}$ to $(\hat{J_{\s}}, \succ)$ is continuous, and the result then follows from a well-known characterisation of an arc. We omit the details.

Now, since $\s$ is admissible, we know that $J_{\s} \ne \emptyset$. We also know, by Proposition~\ref{prop:unbounded}, that each component of $J_{\s}$ is unbounded. Uniqueness then follows from the fact that $\infty$ is the maximal element of $(\hat{J_{\s}}, \succ)$.
%
%
%
\end{proof}
Note that \eqref{it:dev1} and \eqref{it:dev3} and are now an immediate consequence of Lemma~\ref{lem: arcs in Js}, together with Observation~\ref{obs:basics}.
%
It remains to show that the uniform escape property \eqref{it:dev2} holds on the components of $J$. In fact, this is a consequence of Lemma~\ref{lem: arcs in Js}, together with the following.
\begin{lemma}
Suppose that $f$ is a generalised exponential function. If $z,w \in J$ have the same external address, then 
\[\
\lim\limits_{k \to \infty} \max\{\operatorname{Re}f^k(z), \operatorname{Re}f^k(w)\}= \infty.
\]
\end{lemma}
\begin{proof}
We omit the proof of this lemma, which is essentially the same as the proof of \cite[Lemma 3.2]{RRRS}, using Proposition~\ref{prop:expansion} to give expansion.
\end{proof}
%
%
%
\section{Cantor bouquets}
\label{sec.bouquet}
In this section we show that $J$ is Cantor bouquet; in other words, we prove part \eqref{Jbouquet} of Theorem~\ref{theo:main}. It was observed in \cite{LasseNada} that the result of \cite{Mayer} holds for \textit{all} Cantor bouquets. Hence part (c) of Theorem~\ref{theo:main} is a direct consequence of this. Note that the arguments in this section are essentially topological, and very similar to those of \cite{brushing}. Accordingly we give only brief details, and refer to that paper for more detailed explanations and definitions.

In fact, we shall construct a so-called one-sided hairy arc. This is a topological object defined as follows (see also \cite{AartsandO} and \cite{brushing}).
\begin{definition}
\label{def:hairyarc}
A \emph{one-sided hairy arc} is a continuum $X$ containing an arc $B$ (called the base of $X$), and a total order $\prec$ on $B$, such that:
\begin{enumerate}
\item The closure of every connected component of $X \setminus B$ is an arc, with exactly one endpoint in $B$. In particular, for each $x \in X \setminus B$, there exists a unique arc $\gamma_x :[0,1] \to X$ such that $\gamma_x(0)=x$, $\gamma_x(t)\notin B$ for $t <1$, and $\gamma_x(1) \in B$. In this case, we say that $x$ belongs to the hair attached at $\gamma_x(1)$.\label{osha-prop1}
 \item All the hairs are attached at the same side of the base.\label{osha-prop2}
 \item Distinct components of $X \setminus B$ have disjoint closures, and $X \setminus B$ is dense in $X$. \label{osha-prop3} 
 \item If $x_0 \in X \setminus B$ and $x_n \in X \setminus B$ is a sequence of points converging to $x_0$, then $\gamma_{x_n} \rightarrow \gamma_{x_0}$ in the Hausdorff metric.\label{osha-prop4}
\item If $b \in B$ and $x$ belongs to the hair attached at $b$, then there exist sequences $x^+_n, x^-_n$, attached respectively at points $b^+
_n, b^-_n \in B$, such that $b^-_n \prec b \prec b^+_n$ and $x^-_n, x^+_n \rightarrow x$ as $n\rightarrow\infty$.\label{osha-prop5}
\end{enumerate}
\end{definition}
It is known that if $X$ is a one-sided hairy arc, then $X \setminus B$ is homeomorphic to a topological object known as a \emph{straight brush}; we omit the definition, which can be found at \cite{AartsandO}. Our goal is to construct a suitable base $B$ so that $J \cup B$ is a one-sided hairy arc. Since a Cantor bouquet is, by definition, a set ambiently homeomorphic to a straight brush, the result follows.

We follow the construction in \cite[Section 5]{brushing}, although our construction is slightly easier since (up to $2 \pi i$ translation we only have one tract. We define $B$ to be the union of;
\begin{itemize}
\item the set $\Z^{\N_0}$ of all external addresses;
\item the set of all so-called ``intermediate external addresses'' obtained by adding an intermediate entry between any pair of integers;
\item the set $\{-\infty, \infty\}$.
\end{itemize}
We then let $\tilde{H}= \overline{H}\cup B$; recall that $H$ is the image of the tracts, defined in \eqref{eq:Hdef}. Exactly as in \cite[Section 5]{brushing} we can define a topology on $\tilde{H}$ by specifying a neighbourhood base for every $\underline{s} \in B$. It then follows from \cite[Proposition 5.6]{brushing},  that $\tilde{H}$ is homeomorphic to the closed unit disc, and $B$ is homeomorphic to an arc. 

First, we show that the set of admissible external addresses (see Definition \ref{def:admissible}) is dense in the set of all external addresses.
\begin{proposition}
\label{prop:dense addresses}
The set of admissible external addresses is dense in $\Z_{\N_0}$.
\end{proposition}
\begin{proof}
We know from Theorem \ref{thm:gbounded-admissible} that $g$-bounded external addresses are admissible. Hence, since periodic external addresses are certainly $g$-bounded, we deduce that periodic external addresses are admissible. The result follows since periodic external addresses are dense in $\Z_{\N_0}$.
\end{proof}
Let $\tilde{J}$ denote the closure of $J$ in the space $\tilde{H}$. Our goal is to show that $\tilde{J}$ is a one-sided hairy arc. To achieve this we need some results which together imply that properties \eqref{osha-prop1}-\eqref{osha-prop5} from Definition \ref{def:hairyarc} hold. 
\begin{proposition}
\label{prop:hairy arc}
The set $\tilde{J}$ is a continuum with $J = \tilde{J} \setminus B$. Moreover, the closure of every component of $J$ is an arc, with exactly one endpoint in $B$, distinct components of $J$ have disjoint closures in $\tilde{J}$, and $J$ is dense in $\tilde{J}$.
\end{proposition}
We know that $B$ is an arc. Note that this proposition gives properties \eqref{osha-prop1} and \eqref{osha-prop3}. Moreover, $\tilde{J}$ is one-sided by construction, hence property \eqref{osha-prop2} is satisfied.  
\begin{proof}[Proof of Proposition~\ref{prop:hairy arc}]
Recall from Lemma~\ref{lem: arcs in Js} that each component of $J$ is a simple closed arc to infinity, $J_{\s}$, for some external address $\s$. Suppose that $\s$ is an admissible external address. We can deduce from the topology on $\tilde{H}$ that points of $J_{\s}$ cannot accumulate on any element of $B$ apart from $\s$. Hence $J_{\s} \cup \{\s\}$ is a compact subset of $\tilde{H}$. Moreover, $J_{\s} \cup \{\s\}$ is connected.

It follows from Proposition \ref{prop:dense addresses} that $B \subset \tilde{J}$. Hence $\tilde{J}$ is the disjoint union 
\begin{equation}
\label{eq:J}
\tilde{J} = \bigcup_{\text{admissible } \s} J_{\s} \cup B,
\end{equation}
where the union is taken over the admissible external addresses. 

$B$ is homeomorphic to an arc, and so connected. Also, $\tilde{H}$ is a compact metric space, and hence so is $\tilde{J}$. The claims of the proposition follow from these facts, together with \eqref{eq:J}.
\end{proof}

In order to prove the accumulation of hairs, i.e., property \eqref{osha-prop5}, we use the following result.
\begin{proposition}
\label{prop:property 6}
Suppose that $z_0 \in J$. Then there are sequences $z^-_n , z^+_n \in J$, with $\addr(z^-_n) < \addr(z_0) < \addr(z^+_n)$, for $n \in\N$, and $z^-_n \rightarrow z_0$, $z^+_n \rightarrow z_0$ as $n \rightarrow\infty$.
\end{proposition}
\begin{proof}
Choose $p \in \N$. Let $U$ be the component of $f^{-p}(H)$ containing $z_0$, and let $\phi : H \to U$ be the inverse to $f^{-p}$. Define a pair of points $z^\pm_p = \phi(f^p(z_0) \pm 2 \pi i)$, so that, by definition, we have $\addr(z^-_p) < \addr(z_0) < \addr(z^+_p)$, for $p \in \N$. It follows by Proposition~\ref{prop:expansion} that $z^\pm_n \rightarrow z_0$ as $n \rightarrow \infty$, as required.
\end{proof}
The following proposition is analogous to \cite[Proposition 6.1]{brushing} and we omit the proof.
\begin{proposition}
\label{prop:6.1 brushing}
Suppose that $a_n \in J$ converges to a point $a \in J$, and that, for each $n$, $b_n \in J_{addr(a_n)}$ has the same external address as $a_n$ and satisfies $b_n \succ a_n$ in the speed ordering of $f$. If $b \in J$ is an accumulation point of the sequence $b_n$, then $b \succ a$.
\end{proposition}
We use Proposition \ref{prop:6.1 brushing} as a tool to prove property \eqref{osha-prop4}, as shown below. 
\begin{proposition}
\label{prop: property 5}
 If $x_0 \in X \setminus B$ and $x_n \in X \setminus B$ is a sequence of points converging to $x_0$, then $\gamma_{x_n} \rightarrow \gamma_{x_0}$ in the Hausdorff metric.
 \end{proposition}
\begin{proof}
Passing to a subsequence, we may assume that $\gamma_{x_n}$ converges in the Hausdorff metric to a limit $\gamma$. Then $\gamma \subset J_{\underline{s}}\cup \{s\},$ where $\underline{s}= \addr(z_0)$. Note that $\gamma$ is connected as the Hausdorff limit of compact connected subsets of the compact space $\tilde{J}$, and also it contains both $z_0$ and $\underline{s}$. Hence we have that $\gamma_{z_0} \subset \gamma.$
It remains to show that $\gamma \subset \gamma_{z_0}$. Note that this inclusion follows from Proposition \ref{prop:6.1 brushing}. 
\end{proof}
We have shown that $\tilde{J} = J \cup B$ is a one-sided hairy arc. Hence, for the reasons noted earlier, $J$ is a Cantor bouquet, which completes the proof of Theorem~\ref{theo:main}.
\bibliographystyle{alpha}
\bibliography{../Exp.References}
\end{document}